\documentclass{article}
\usepackage{amssymb,amsmath,amsthm, mathrsfs}
\usepackage{graphicx, float,color}
\usepackage[encapsulated]{CJK}
\usepackage[english]{babel}
\usepackage[colorlinks=true]{hyperref}
\hypersetup{urlcolor=blue, linkcolor=red, citecolor=blue}
\usepackage{geometry} % to change the page dimensions
\usepackage{array} % for better arrays (eg matrices) in maths
\usepackage{paralist} % very flexible & customisable lists (eg. enumerate/itemize, etc.)
\usepackage{verbatim} % adds environment for commenting out blocks of text & for better verbatim
\usepackage[all]{xy}
\newtheorem{theorem}{Theorem}[section]

\newtheorem {proposition}[theorem]{Proposition}

\newtheorem {definition}[theorem]{Definition}
\newtheorem {example}[theorem]{Example}

\usepackage{sectsty}
\allsectionsfont{\sffamily\mdseries\upshape} % (See the fntguide.pdf for font help)

\title{Disjoint strong transitive operators on Hilbert $C^{\ast}$-modules}

\author{Song-Ung Ri, Hyon-Hui Ju*}

\date{}
\begin{document}
\maketitle
{\makeatletter\renewcommand*\@makefnmark{}\footnotetext

{*Corresponding author E-mail address:
hh.ju0524@ryongnamsan.edu.kp(Hyon-Hui Ju) }\makeatother}
	
\centerline{Faculty of Mathematics, {\bf Kim Il Sung} University,}
	
\centerline{Pyongyang, Democratic People's Republic of Korea}
%\selectlanguage{english} %%% remove comment delimiter ('%') and select language if required
	
\begin{abstract}
This paper is about disjoint $\mathcal{F}$-transitive operators on Hilbert $C^{\ast}$-modules, where $\mathcal{F}$ is a finitely invariant Furstenberg family of subsets of $\mathbb{N}$.
We characterize disjoint $\mathcal{F}$-transitivity for generalized bilateral weighted shift operators on the standard Hilbert module over $C^{\ast}$-algebra of compact operators on the separable Hilbert space. Then we give a sufficient condition for these operators to be disjoint chaotic and give concrete examples. We also characterize disjoint $\mathcal{F}$-transitivity for generalized translation operators on the non-unitary $C^{\ast}$-algebra and provide some applications.		
\end{abstract}
	

\section{Introduction}
In this paper, we study disjoint $\mathcal{F}$-transitivity of the operators introduced by Ivkovi\'{c} \cite{I24}, where $\mathcal{F}$ is a finitely invariant Furstenberg family of subsets of $\mathbb{N}$.

In \cite{I24} the author characterized topological transitivity of the generalized bilateral weighted shift operators on the standard Hilbert module over $C^{\ast}$-algebra of compact operators on the separable Hilbert space and obtained a sufficient condition for chaos. And he characterized topological transitivity of the generalized translation operators on the non-unitary $C^{\ast}$-algebra.

Here, we discuss disjoint dynamical properties of the operators introduced in \cite{I24}.

Disjoint dynamics of linear operators have been studied a lot for the last few decades after Bernal \cite{B07} and B\`{e}s and Peris \cite{BP07} introduced the notion of disjoint hypercyclicity independently for the first time in 2007(see \cite{B07}, \cite{BMS14}, \cite{BP07}, \cite{CK23}, \cite{CKP18}, \cite{COT20}, \cite{IT23dis}, \cite{KBA22}).
Beranl \cite{B07} provided two criteria for disjoint hypercyclicity of the operators acting on topological vector spaces and B\`{e}s and Peris \cite{BP07} provided the Disjoint Hypercyclicity Criterion for operators acting on a Fr\'{e}chet space.
In 2014, B\`{e}s, Martin and Sanders \cite{BMS14} studied disjoint hypercyclicity of weighted shifts and they showed that any finite collections of disjoint hypercyclic shifts never satisfy the disjoint hypercyclicity criterion.

In 2023 Chen and Kosti\'{c} \cite{CK23} characterized disjoint topological transitivity and disjoint chaos for weighted translations on Orlicz spaces.
In 2023 Ivkovi\'{c} and Tabatabaie \cite{IT23dis} characterized disjoint topological transitivity for elementary operators.

In \cite{RJ26} we characterized disjoint topological transitivity of the generalized bilateral weighted shift operators on the standard Hilbert module over $C^{\ast}$-algebra of compact operators on the separable Hilbert space, which were introduced in \cite{I24}.

In \cite{KBA22}, the authors characterized disjoint $\mathcal{F}$-transitivity of composition operators on the space $H(\Omega)$, where $\mathcal{F}$ is a Furstenberg family of subsets of $\mathbb{N}$.

In this paper, we study disjoint $\mathcal{F}$-transitivity and disjoint chaos for the generalized bilateral weighted shift operators on the standard Hilbert module over $C^{\ast}$-algebra of compact operators on the separable Hilbert space and disjoint $\mathcal{F}$-transitivity for the generalized translation operators on the non-unitary $C^{\ast}$-algebra. 
This paper is organized as follows:
In Section 3.1 we provide a necessary and sufficient condition to be disjoint $\mathcal{F}$-transitive, where $\mathcal{F}$ is a finitely invariant Furstenberg family, and a sufficient condition to be disjoint chaotic for the generalized bilateral weighted shift operators on the standard Hilbert module over $C^{\ast}$-algebra of compact operators on the separable Hilbert space. Then we give an example for disjoint chaos and disjoint topologically mixing of these operators.
In Section 3.2 we provide a necessary and sufficient condition to be disjoint $\mathcal{F}$-transitive, where $\mathcal{F}$ is a finitely invariant Furstenberg family, for the generalized translation operators on the non-unitary $C^{\ast}$-algebra.

\section{Preliminaries}
Let $X$ denote a separable Banach space and $\mathcal{L}(X)$ be the algebra of continuous linear operators on $X$.
An operator $T$ in $\mathcal{L}(X)$ is called \textit{topologically transitive} (\textit{topologically mixing}) if for every pair of non-empty open sets $U, V\subset X$, the return set $N(U,V):=\{n\in\mathbb{N}: T^{n}U\cap V\neq \emptyset\}$ is infinite (cofinite).
We say that an operator $T$ in $\mathcal{L}(X)$ is \textit{hypercyclic} if there exists a vector $x \in X$, called a \textit{hypercyclic vector}, such that its orbit $ Orb(x, T) =\{T^{n}x:n\in {\mathbb N} \}$ is dense in $X$. This is equivalent to demanding the existence of a vector $x \in X$ such that for any non-empty open set $U \subset X$ the return set $N(x,U) := \{n \in {\mathbb{N}}:T^{n}x \in U \}$ is non-empty. According to Birkhoff transitivity theorem, an operator $T$ on a separable Banach space $X$ is hypercyclic if and only if it is topologically transitive (Theorem 1.16, \cite{GM11}). A hypercyclic operator $T$ in $\mathcal{L}(X)$ with a dense set of periodic points  is called \textit{chaotic}.

Disjointness in hypercyclicity was introduced independently by Bernal \cite{B07}, and B\`{e}s and Peris \cite{BP07} in 2007.
For $N\geq 2$, the operators $T_{1}$,...,$T_{N}\in \mathcal{L}(X)$ are called \textit{disjoint hypercyclic} if the direct sum $T_{1}\oplus \cdots \oplus T_{N}$ has a hypercyclic vector of the form $(x,...,x)\in X^{N}$. Such a vector $x\in X$ is called a \textit{disjoint hypercyclic vector} for $T_{1},...,T_{N}$. If the set of disjoint hypercyclic vectors is dense in $X$, then the operators are called \textit{densely disjoint hypercyclic}.

And for $N\geq 2$, operators $T_{1}$,...,$T_{N}\in \mathcal{L}(X)$ are called \textit{disjoint topologically transitive}(\textit{disjoint topologically mixing}) if for any non-empty open subsets $U, V_{1}, V_{2},...,V_{N}$ of $X$, the return set $N(U,V_{1}, ..., V_{N}) :=\{n\in\mathbb{N}: U\cap T^{-n}_{1}(V_{1})\cap\cdots\cap T^{-n}_{N}(V_{N})\neq \emptyset\}$ is infinite (cofinite).
The operators $T_{1}$,...,$T_{N}\in \mathcal{L}(X)$ are disjoint topologically transitive if and only if the set of disjoint hypercyclic vectors is dense in $X$ (Proposition 2.3, \cite{BP07}).

Disjointness in chaos was initially investigated in \cite{CKP18} and formulated in \cite{CK23}.

\begin{definition}\textnormal{\cite{CK23}
Let $X$ be a separable infinite-dimensional Banach space. For $N\geq2$, operators $T_{1}$, $T_{2}$, ... , $T_{N}$ on $X$ are \textit{disjoint chaotic} if they are disjoint topologically transitive and the set of periodic elements, denoted by $\mathcal{P}(T_{1}, T_{2}, \cdots, T_{N})=\{(x_{1},x_{2},\cdots,x_{N})\in X^{N}:\exists n\in\mathbb{N}, (T^{n}_{1}x_{1},T^{n}_{2}x_{2},\cdots,T^{n}_{N}x_{N})=(x_{1},x_{2},\cdots,x_{N})\}$, is dense in $X^{N}$.
}\end{definition}

From the definition of disjoint chaos, it is trivial to see that the operators $T_{1}$, $T_{2}$, ... , $T_{N}$ on $X$ are disjoint chaotic if and only if they are disjoint transitive and each of them is chaotic.

A collection $\mathcal{F}$ of subsets of $\mathbb{N}$ is called a \textit{Furtenberg family} if it is hereditary upward, that is, $A\in\mathcal{F}$ and $A\subset B$ imply $B\in \mathcal{F}$.
We say that a Furstenberg family is \textit{proper} if it is non-empty and does not contain the empty set.
And a Furstenberg family is called \textit{finitely invariant} if for any $A\in\mathcal{F}$ and all $n\geq 0$, $A\setminus [0,n]\in\mathcal{F}$.
The family of all sets containing infinitely (cofinitely) many positive integers, is denoted by $\mathcal{F}_{\rm inf}$ ($\mathcal{F}_{\rm cof}$).

Now we recall the notion of topological transitivity and disjoint topological transitivity via Furstenberg families.
\begin{definition}\textnormal{\cite{BMPP19}
Let $X$ be a separable infinite-dimensional Banach space and $\mathcal{F}$ be a proper Furstenberg family.
An  operator $T$ on $X$ is called \textit{$\mathcal{F}$-transitive} if for any non-empty open sets $U,V\subset X$, the return set $N_{T}(U,V):=\{n\in\mathbb{N}: T^{n}U\cap V\neq \emptyset\}\in \mathcal{F}$.
}\end{definition}

\begin{definition}\textnormal{ \cite{KBA22}
Let $X$ be a separable infinite-dimensional Banach space and $\mathcal{F}$ be a proper Furstenberg family.
For $N\geq 2$, operators $T_{1}$, $T_{2}$, ..., $T_{N}$ on $X$ is called \textit{disjoint $\mathcal{F}$-transitive} if for any non-empty open sets $U, V_{1}, V_{2},...,V_{N}$ of $X$ , the return set $N(U,V_{1}, ..., V_{N}) :=\{n\in\mathbb{N}: U\cap T^{-n}_{1}(V_{1})\cap\cdots\cap T^{-n}_{N}(V_{N})\neq \emptyset\}\in \mathcal{F}$
}\end{definition}

\section{Main results}
\subsection{Generalized bilateral weighted shift operators on the standard Hilbert $C^{\ast}$-modules}
In this subsection we study disjoint dynamics the generalized bilateral weighted shift operators on the standard right Hilbert module over  $C^{\ast}$-algebra of compact operaters on Hilbert space. Recall the definition of the generalized bilateral weighted shift operators in \cite{I24}.

Let $\mathcal{H}$ be a separable Hilbert space and $B(\mathcal{H})$ be the $C^{\ast}$-algebra of all bounded linear operators on $\mathcal{H}$. And let $\mathcal{A_{C}}:=B_{0}(\mathcal{H})$ be the $C^{\ast}$-algebra of all compact operators on $\mathcal{H}$. Assume that $W:=(W_{j})_{j\in\mathbb{Z}}$ is a uniformly bounded sequence of invertible operators in $B(\mathcal{H})$ such that the sequence $\{W^{-1}_{j}\}_{j\in\mathbb{Z}}$ is also uniformly bounded in $B(\mathcal{H})$ and $U$ is a unitary operator on $\mathcal{H}$. Generalized bilateral weighted shift operator $T_{U, W}$ on $\ell_{2}(\mathcal{A_{C}})$, the standard right Hilbert module over $\mathcal{A_{C}}$, is defined by
\[(T_{U,W}(x))_{n}=W_{n}x_{n-1}U\]
for all $n\in\mathbb{Z}$ and $x:=(x_{j})_{j\in\mathbb{Z}}\in\ell_{2}(\mathcal{A_{C}})$.

$T_{U,W}$ is invertible and its inverse $S_{U,W}$ is given by
\[(S_{U,W}(y))_{n}=W^{-1}_{n+1}y_{n+1}U^{\ast}\]
for all $n\in\mathbb{Z}$ and $y:=(y_{j})_{j\in\mathbb{Z}}\in\ell_{2}(\mathcal{A_{C}})$.
Then we can check that
\[(T^{n}_{U,W}(x))_{i}=W_{i}W_{i-1}...W_{i-n+1}x_{i-n}U^{n}\]
and
\[(S^{n}_{U,W}(y))_{i}=W^{-1}_{i+1}W^{-1}_{i+2}...W^{-1}_{i+n}y_{i+n}U^{\ast n}\]
for all $n\in\mathbb{N}$, $i\in\mathbb{Z}$ and $x:=(x_{i})_{i}$, $y:=(y_{i})_{i}\in \ell_{2}(\mathcal{A_{C}})$. For more details see \cite{I24} and the references therein.

Throughout this paper, we denote $[J]:=\{-J, -J+1,..., J\}$ for each $J\in\mathbb{N}$, $L_{m}:=\textrm{span}\{e_{-m}, e_{-m+1},..., e_{m}\}$ where $\{e_{i}\}_{i\in\mathbb{Z}}$ is an orthonormal basis for $\mathcal{H}$ and $P_{m}$ is the orthogonal projection onto $L_{m}$.

In the following theorem we give a necessary and sufficient condition for the generalized bilateral weighted shift operaters on $\ell_{2}(\mathcal{A_{C}})$ to be disjoint $\mathcal{F}$-transitive.
\begin{theorem}\label{the3.1}
Let $W^{(1)}$, $W^{(2)}$,...,$W^{(N)}$ be uniformly bounded sequences of invertible operators in $B(\mathcal{H})$ and $U^{(1)}$, $U^{(2)}$,...,$U^{(N)}$ be unitary operators in $B(\mathcal{H})$ such that for each $m\in \mathbb{N}$ there exists $N_{m}\in\mathbb{N}$ such that for each $n\geq N_{m}$, each pair of distinct $s,l\in \{1,2,,,N\}$
\[{U^{(s)}}^{n}{U^{(l)}}^{-n}(L_{m})\perp L_{m}\textrm{    }\textrm{    }\textrm{    }(\ast).\]
Then the following are equivalent.
	
$(1)$ The operators $T_{U^{(1)},W^{(1)}}$, $T_{U^{(2)},W^{(2)}}$,...,$T_{U^{(N)},W^{(N)}}$ are disjoint $\mathcal{F}$-transitive.
	
$(2)$ For every $J,m \in\mathbb{N}$, there exist sequences
$\{D^{(k)}_{j}\}^{\infty}_{k=1}$ and $\{G^{(k)}_{1,j}\}^{\infty}_{k=1}$,...,$\{G^{(k)}_{N,j}\}^{\infty}_{k=1}$, for all $j\in [J]$, satisfying;
for every $\varepsilon>0$ there exists a set $F_{\varepsilon}\in\mathcal{F}$ such that for any $k\in F_{\varepsilon}$ and any $j\in [J]$, $1\leq l\neq s \leq N$,
\[\|D^{(k)}_{j}-P_{m}\|<\varepsilon, \|G^{(k)}_{l,j}-P_{m}\|<\varepsilon,\]
\[\|W^{(l)}_{j+k}W^{(l)}_{j+k-1}\cdots W^{(l)}_{j+1}D^{(k)}_{j}\|<\varepsilon,
\|{W^{(l)}_{j-k+1}}^{-1}{W^{(l)}_{j-k+2}}^{-1}\cdots{W^{(l)}_{j}}^{-1}G^{(k)}_{l,j}\|<\varepsilon,\]
and
\[\|W^{(s)}_{j}\cdots W^{(s)}_{j-k+1}{W^{(l)}_{j-k+1}}^{-1}{W^{(l)}_{j-k+2}}^{-1}\cdots {W^{(l)}_{j}}^{-1}G^{(k)}_{l,j}\| <\varepsilon.\]
\end{theorem}

\begin{proof}
	First we show $(1)\Rightarrow(2)$. Assume that the operators $T_{U^{(1)},W^{(1)}}$, $T_{U^{(2)},W^{(2)}}$,...,$T_{U^{(N)},W^{(N)}}$ are disjoint $\mathcal{F}$-transitive.
	Let $J,m\in\mathbb{N}$ and define $x=(x_{j})_{j}\in \ell_{2}{(\mathcal{A_{C}})}$ by $x_{j}:=P_{m}$ for all $j\in[J]$ and $x_{j}:=0$ for all $j\in \mathbb{Z}\setminus[J]$.
	Then there exists a sequence $\{y^{(k)}\}_{k}\subset \ell_{2}(\mathcal{A_{C}})$ such that
	\[(y^{(k)}, T^{k}_{U^{(1)},W^{(1)}}(y^{(k)}), ..., T^{k}_{U^{(N)},W^{(N)}}(y^{(k)}))\stackrel{\mathcal F}{\longrightarrow}(x,x, ..., x),\]
	as $n\rightarrow\infty$, i.e. for every $\varepsilon>0$, there exists a set $F_{\varepsilon}\in\mathcal{F}$ such that for any $k\in F_{\varepsilon}$, $\|y^{(k)}-x\|\ <\varepsilon$ and $\|T^{k}_{U^{(l)},W^{(l)}}(y^{(k)})-x\|<\varepsilon$ for each $l\in \{1,2,...,N\}$.
	
	Since $\mathcal{F}$ is finitely invariant, we may assume that for every $k\in F_{\varepsilon}$, $k>2J$.
	
	Set
	\[D^{(k)}_{j}:=y^{(k)}_{j}P_{m}, G^{(k)}_{l,j}:=W^{(l)}_{j}W^{(l)}_{j-1}...W^{(l)}_{j-k+1}y^{(k)}_{j-k}{U^{(l)}}^{k}P_{m}\]
	for each $j\in[J]$ and $l\in \{1,2,...,N\}$. 
	Then we have for $j\in[J]$ and $l\in\{1,2,...,N\}$,
	
	\begin{eqnarray}\begin{split}\nonumber
	\|D^{(k)}_{j}-P_{m}\|
	&=\|y^{(k)}_{j}P_{m}-P_{m}\|\leq \|y^{(k)}_{j}-P_{m}\|\|P_{m}\|
	\\
	&\leq\|y^{(k)}-x\|<\varepsilon,
	\end{split}\end{eqnarray}
	\begin{eqnarray}\begin{split}\nonumber
	\|G^{(k)}_{l,j}-P_{m}\|
	&=\|W^{(l)}_{j}W^{(l)}_{j-1}...W^{(l)}_{j-k+1}y^{(k)}_{j-k}{U^{(l)}}^{k}P_{m}-P_{m}\|
	\\
	&=\|(T^{k}_{U^{(l)},W^{(l)}}(y^{(k)}))_{j}P_{m}-P_{m}\|\leq \|(T^{k}_{U^{(l)},W^{(l)}}(y^{(k)}))_{j}-P_{m}\|
	\\
	&\leq\|T^{k}_{U^{(l)},W^{(l)}}(y^{(k)})-x\|<\varepsilon,
	\end{split}\end{eqnarray}
	\begin{eqnarray}\begin{split}\nonumber
	\|W^{(l)}_{j+k}W^{(l)}_{j+k-1}\cdots W^{(l)}_{j+1}D^{(k)}_{j}\|
	&=\|(T^{k}_{U^{(l)},W^{(l)}}(y^{(k)}))_{j+k}{U^{(l)}}^{-k}P_{m}\|
	\\
	&\leq\|(T^{k}_{U^{(l)},W^{(l)}}(y^{(k)}))_{j+k}\|
	\\
	&\leq\|T^{k}_{U^{(l)},W^{(l)}}(y^{(k)})-x\|<\varepsilon,
	\end{split}
	\end{eqnarray}
	\begin{eqnarray}\begin{split}\nonumber
	\|{W^{(l)}_{j-k+1}}^{-1}{W^{(l)}_{j-k+2}}^{-1}\cdots {W^{(l)}_{j}}^{-1}G^{(k)}_{l,j}\|
	&=\|y^{(k)}_{j-k}{U^{(l)}}^{k}P_{m}\|\leq\|y^{(k)}_{j-k}\|
	\\
	&\leq \|y^{(k)}-x\|<\varepsilon.
	\end{split}
	\end{eqnarray}
	
	And since for each $k\geq N_{m}$ and each pair of distinct $s,l\in \{1,2,,,N\}$,  ${U^{(s)}}^{k}{U^{(l)}}^{-k}(L_{m})\perp L_{m}$, we have that for each pair of distinct $s,l\in \{1,2,,,N\}$,
	\[P_{m}{U^{(s)}}^{n}{U^{(l)}}^{-n}P_{m}=0.\]
	Then we have 
	\begin{eqnarray}\begin{split}\nonumber
	\|W^{(s)}_{j}&\cdots W^{(s)}_{j-k+1}{W^{(l)}_{j-k+1}}^{-1}{W^{(l)}_{j-k+2}}^{-1}\cdots
	{W^{(l)}_{j}}^{-1}G^{(k)}_{l,j}\|
	\\
	&=\|(T^{k}_{U^{(s)},W^{(s)}}(y^{(k)}))_{j} {U^{(s)}}^{-k}{U^{(l)}}^{k}P_{m}\|
	\\
	&=\|(T^{k}_{U^{(s)},W^{(s)}}(y^{(k)}))_{j} {U^{(s)}}^{k}{U^{(l)}}^{k}P_{m}-P_{m}{U^{(s)}}^{-k}{U^{(l)}}^{k}P_{m}\|
	\\
	&\leq\|(T^{k}_{U^{(s)},W^{(s)}}(y^{(k)}))_{j}-P_{m}\|\leq \|T^{k}_{U^{(s)},W^{(s)}}(y^{(k)})-x\|<\varepsilon
	\end{split}
	\end{eqnarray}
	which completes the proof of $(1)\Rightarrow(2)$.
	
	Next, we show $(2)\Rightarrow(1)$. Let $O, V_{1}, V_{2},..., V_{N}$ be non-empty open subsets of $\ell_{2}(\mathcal{A_{C}})$. Assume that $F$ denotes the set of all elements $x=(x_{j})_{j}\in \ell_{2}(\mathcal{A_{C}})$ such that for some $J,m\in \mathbb{N}$, $x_{j}=P_{m}x_{j}$ for all $j\in [J]$ and $x_{j}=0$ for all $j\in \mathbb{Z}\setminus [J]$. Since $F$ is dense in $\ell_{2}(\mathcal{A_{C}})$, we can find some $x=(x_{j})_{j}\in O$ and $y^{(1)}=(y^{(1)}_{j})_{j}\in V_{1}$, . . . , $y^{(N)}=(y^{(N)}_{j})_{j}\in V_{N}$ and sufficiently large $J, m\in\mathbb{N}$ such that $x_{j}=P_{m}x_{j}$ and $y^{(l)}_{j}=P_{m}y^{(l)}_{j}$ for all $l\in\{1,2,...,N\}, j\in[J]$ and $x_{j}=0$ and $y^{(l)}_{j}=0$ for all $l\in\{1,2,...,N\}, j\in \mathbb{Z}\setminus[J]$. Let the sequences $\{D^{(k)}_{j}\}^{\infty}_{k=1}$ and $\{G^{(k)}_{1,j}\}^{\infty}_{k=1}$,...,$\{G^{(k)}_{N,j}\}^{\infty}_{k=1}$ satisfy (ii) for these $J,m\in\mathbb{N}$. For each $k$, define $u^{(k)}, v^{(k)}_{1}, ..., v^{(k)}_{N}\in \ell_{2}(\mathcal{A_{C}})$ by $(u^{(k)})_{j}:=D^{(k)}_{j}x_{j}$ for $j\in[J]$, $(u^{(k)})_{j}:=0$ for $j\in \mathbb{Z}\setminus [J]$, for each $l\in\{1,2,...,N\}$, $(v^{(k)}_{l})_{j}:=G^{(k)}_{l, j}y^{(l)}_{j}$ for $j\in[J]$, $(v^{(k)}_{l})_{j}:=0$ for $j\in \mathbb{Z}\setminus [J]$.
	
	Set
	\[\varphi_{k}:=u^{(k)}+\sum^{N}_{l=1}S^{n_{k}}_{U^{(l)},W^{(l)}}v^{(k)}_{l}.\]
	Now it is sufficient to show that 
	\[(\varphi_{k}, T^{k}_{U^(1),W^(1)}(\varphi_{k}), ..., T^{k}_{U^{(N)},W^{(N)}}(\varphi_{k})) \stackrel{\mathcal F}{\longrightarrow}(x,y_{1}, ..., y_{N}) \textrm{ as } k\rightarrow\infty.\]
	Let $\sum_{j\in[J]}\|x_{j}\|+\sum^{N}_{l=1}\sum_{j\in[J]}\|y^{(l)}_{j}\|=c$.
	From the condition (2), for $\varepsilon>0$ there exists a set $F_{\varepsilon}\in\mathcal{F}$ such that for any $k\in F_{\varepsilon}$ and any $j\in [J]$, $1\leq l\neq s \leq N$,
	\[\|D^{(k)}_{j}-P_{m}\|<\varepsilon/c, \|G^{(k)}_{l,j}-P_{m}\|<\varepsilon/c,\]
	\[\|W^{(l)}_{j+k}W^{(l)}_{j+k-1}\cdots W^{(l)}_{j+1}D^{(k)}_{j}\|<\varepsilon/c,
	\|{W^{(l)}_{j-k+1}}^{-1}{W^{(l)}_{j-k+2}}^{-1}\cdots {W^{(l)}_{j}}^{-1}G^{(k)}_{l,j}\|<\varepsilon/c,\]
	and
	\[\|W^{(s)}_{j}\cdots W^{(s)}_{j-k+1}{W^{(l)}_{j-k+1}}^{-1}{W^{(l)}_{j-k+2}}^{-1}\cdots {W^{(l)}_{j}}^{-1}G^{(k)}_{l,j}\| <\varepsilon/c.\]
	Then
	\begin{eqnarray}\begin{split}\nonumber
	\|\varphi_{k}-x\|
	&\leq \|u^{(k)}-x\|+ \sum^{N}_{l=1}\|S^{k}_{U^{(l)},W^{(l)}}v^{(k)}_{l}\|
	\\
	&\leq\sum_{j\in\mathbb{Z}}\|u^{(k)}_{j}-x_{j}\|+ \sum^{N}_{l=1}\sum_{j\in\mathbb{Z}}\|(S^{k}_{U^{(l)},W^{(l)}}v^{(k)}_{l})_{j}\|
	\\
	&=\sum_{j\in[J]}\|u^{(k)}_{j}-x_{j}\|+ \sum^{N}_{l=1}\sum_{j\in[J]}\|(S^{k}_{U^{(l)},W^{(l)}}v^{(k)}_{l})_{j-k}\|
	\\
	&=\sum_{j\in[J]}\|D^{(k)}_{j}x_{j}-P_{m}x_{j}\|+ \sum^{N}_{l=1}\sum_{j\in[J]}\|{W^{(l)}_{j-k+1}}^{-1}\cdots {W^{(l)}_{j}}^{-1} G^{(k)}_{l,j}y^{(l)}_{j}\|
	\\
	&\leq \sum_{j\in[J]}\|D^{(k)}_{j}-P_{m}\|\|x_{j}\|+ \sum^{N}_{l=1}\sum_{j\in[J]}\|{W^{(l)}_{j-k+1}}^{-1}\cdots {W^{(l)}_{j}}^{-1} G^{(k)}_{l,j}\| \|y^{(l)}_{j}\|
	\\
	&<\varepsilon,
	\end{split}
	\end{eqnarray}
	\begin{eqnarray}\begin{split}\nonumber
	\|T^{k}_{U^{(l)},W^{(l)}}(\varphi_{k})-y^{(l)}\| &\leq\|T^{k}_{U^{(l)},W^{(l)}}(u^{(k)})\|+\|v^{(k)}_{l}-y^{(l)}\|
	\\
	& +\sum^{N}_{s\neq l, s=1} \|T^{k}_{U^{(l)},W^{(l)}}(S^{k}_{U^{(s)},W^{(s)}}(v^{(k)}_{s}))\|
	\\
	&\leq \sum_{j\in\mathbb{Z}}\|(T^{k}_{U^{(l)},W^{(l)}}(u^{(k)}))_{j}\|
	+\sum_{j\in\mathbb{Z}}\|(v^{(k)}_{l})_{j}-y^{(l)}_{j}\| 
	\\
	& +\sum^{N}_{s\neq l, s=1} \sum_{j\in\mathbb{Z}} \|(T^{k}_{U^{(l)},W^{(l)}}(S^{k}_{U^{(s)},W^{(s)}}(v^{(k)}_{s})))_{j}\|
	\\
	&=\sum_{j\in[J]}\|W^{(l)}_{j+k}W^{(l)}_{j+k-1}\cdots W^{(l)}_{j+1}D^{(k)}_{j}x_{j}\|+\sum_{j\in[J]}\|G^{(k)}_{l,j}y^{(l)}_{j}-P_{m}y^{(l)}_{j}\|
	\\
	& +\sum^{N}_{s\neq l, s=1} \sum_{j\in[J]} \|W^{(l)}_{j}\cdots W^{(l)}_{j-k+1}{W^{(s)}_{j-k+1}}^{-1}{W^{(s)}_{j-k+2}}^{-1}\cdots
	{W^{(s)}_{j}}^{-1} G^{(k)}_{s,j}y^{(s)}_{j}\|
	\\
	&\leq \sum_{j\in[J]}\|W^{(l)}_{j+k}W^{(l)}_{j+k-1}\cdots W^{(l)}_{j+1}D^{(k)}_{j}\|\|x_{j}\|+\sum_{j\in[J]}\|G^{(k)}_{l,j}-P_{m}\| \|y^{(l)}_{j}\|
	\\
	& +\sum^{N}_{s\neq l, s=1}\sum_{j\in[J]}\|W^{(l)}_{j}\cdots W^{(l)}_{j-k+1}{W^{(s)}_{j-k+1}}^{-1}{W^{(s)}_{j-k+2}}^{-1}\cdots
	{W^{(s)}_{j}}^{-1}G^{(k)}_{s,j}\|\|y^{(s)}_{j}\|
	\\
	&<\varepsilon.
	\end{split}
	\end{eqnarray}
	This concludes $(2)\Rightarrow(1)$.
\end{proof}

In Theorem \ref{the3.1} we get a necessary and sufficient condition for the operators $T_{U^{(1)},W^{(1)}}$, $T_{U^{(2)},W^{(2)}}$, ..., $T_{U^{(N)},W^{(N)}}$ to be disjoint topologically transitive in case that the finitely invariant Furstanberg family $\mathcal{F}$ is $\mathcal{F}_{\inf}$, which is the result of Theorem 3.4 in \cite{RJ26}. Furthermore, we get an equivalent condition for the operatrs $T_{U^{(1)},W^{(1)}}$, $T_{U^{(2)},W^{(2)}}$,...,$T_{U^{(N)},W^{(N)}}$ to be disjoint topologically mixing by taking the Fursatenberg family $\mathcal{F}$ by $\mathcal{F}_{\textnormal{cof}}$.  

The following proposition provides a sufficient condition for genralized bilateral wighted shift operators to be disjoint chaotic.

\begin{proposition}
Let $W^{(1)}$, $W^{(2)}$,...,$W^{(N)}$ be uniformly bounded sequences of invertible operators in $B(\mathcal{H})$ and $U^{(1)}$, $U^{(2)}$,...,$U^{(N)}$ be unitary operators in $B(\mathcal{H})$ such that for each $m\in \mathbb{N}$ there exists an $N_{m}\in\mathbb{N}$ such that for each $n\geq N_{m}$, each pair of distinct $s,l\in \{1,2,,,N\}$
\[{U^{(s)}}^{n}{U^{(l)}}^{-n}(L_{m})\perp L_{m}\textrm{    }\textrm{    }\textrm{    }(\ast).\]
Then $(2)\Rightarrow(1)$.

$(1)$ The operators $T_{U^{(1)},W^{(1)}}$, $T_{U^{(2)},W^{(2)}}$,...,$T_{U^{(N)},W^{(N)}}$ are disjoint chaotic.

$(2)$ For every $J,m \in\mathbb{N}$, there exist a strictly increasing sequence $\{n_{k}\}_{k}\subset \mathbb{N}$ and a sequence
$\{D^{(k)}_{i}\}^{\infty}_{k=1}$ for all $i\in [J]$, satisfying;
for $j\in [J]$, $1\leq l\neq s \leq N$,
\[\lim_{k\rightarrow\infty}\|D^{(k)}_{j}-P_{m}\|=0,\]
\[\lim_{k\rightarrow\infty}\|W^{(s)}_{j}\cdots W^{(s)}_{j-n_{k}+1}{W^{(l)}_{j-n_{k}+1}}^{-1}{W^{(l)}_{j-n_{k}+2}}^{-1}\cdots {W^{(l)}_{j}}^{-1}D^{(k)}_{j}\|=0 \]

and

$(i)$ $\sum^{\infty}_{t=1}\|W^{(i)}_{j+tn_{k}}\ldots W^{(i)}_{j+1}D^{(k)}_{j}\|^{2}$ converges for all $1\leq i\leq N$, $j\in[J]$ and $k\in \mathbb{N}$,

$(ii)$ $\sum^{\infty}_{t=1}\|{W^{(i)}_{j-tn_{k}+1}}^{-1}\ldots {W^{(i)}_{j}}^{-1}D^{(k)}_{j}\|^{2}$ converges for all $1\leq i\leq N$, $j\in[J]$ and $k\in \mathbb{N}$.
\end{proposition}
\begin{proof}
It is obvious from Proposition 3.1 and Theorem 3.4 in \cite{RJ26}.
\end{proof}
It is not difficult to find the generalized bilateral weighted shift operaters that are disjoint chaotic. In fact, the operators, which were provided in Example 3.6 of \cite{RJ26}, are disjoint chaotic and disjoint topologically mixing. For reader convenience, we will repeat it here and give a (short) proof for disjoint chaos. Before providing the example for disjoint chaos, we briefly summarize an example for chaos in the following. 

\begin{example}\label{ex3.1}\textnormal{
Let $\mathcal{H}$ be a separable Hilbert space and $(e_{j})_{j\in\mathbb{Z}}\subset\mathcal{H}$ be an orthonormal basis for $\mathcal{H}$. Let $\alpha>1$ be a real number. And $V$ is a bounded linear operator on $\mathcal{H}$ defined by}
\begin{displaymath}
V(e_{i})=\left\{ \begin{array}{ll}
\alpha e_{i+1}, & \textrm{for } i<0\\
\frac{1}{\alpha}e_{i+1}, & \textrm{for } i\geq 0.
\end{array} \right.
\end{displaymath}

\textnormal{
Since for any $n\geq 0$ and any $x=\sum_{i\in\mathbb{Z}}x_{i}e_{i}\in \mathcal{H}$,
\[V^{n}(x)=\sum_{i<-n}\alpha^{n}x_{i}e_{i+n}+\sum_{-n\leq i< 0}\alpha^{-n-2i}x_{i}e_{i+n}+\sum_{i\geq 0}\alpha^{-n}x_{i}e_{i+n},\]
we get that for $n\geq m$ and any $x=\sum_{i\in\mathbb{Z}}x_{i}e_{i}\in \mathcal{H}$,
\[\|V^{n}P_{m}(x)\|=\|\sum_{-m\leq i< 0}\alpha^{-n-2i}x_{i}e_{i+n}+\sum_{0\leq i\leq m}\alpha^{-n}x_{i}e_{i+n}\|,\]  and thus $\|V^{n}P_{m}\|=\alpha^{-n+2m}$.
Now, let us define $W_{j}=V$ for every $j\in \mathbb{Z}$. Then for any $k\in \mathbb{N}$ and any $m\in\mathbb{N}$, 
\[\sum^{\infty}_{t=1}\|W_{j+tk}W_{j+tk-1},...,W_{j+1}P_{m}\|^{2}= \sum^{\infty}_{t=1}\|V^{tk}P_{m}\|^{2}\]
\[=\sum^{m-1}_{t=1}\|V^{tk}P_{m}\|^{2}+\sum^{\infty}_{t=m}\|V^{tk}P_{m}\|^{2}\]
\[=\sum^{m-1}_{t=1}\|V^{tk}P_{m}\|^{2}+\sum^{\infty}_{t=m}\alpha^{-2tk+4m}.\]
Since $\alpha>1$, the above series converges.
As a similar way, we obtain that the series
$\sum^{\infty}_{t=1}\|W^{-1}_{j-tk+1}\ldots W^{-1}_{j}P_{m}\|^{2}$ also converges for any $k\in \mathbb{N}$. 
Therefore, the operator $T_{U,W}$ is chaotic from Proposition 3.1 of \cite{RJ26}.}
\end{example}

The above operator given in Example \ref{ex3.1} was already provided in \cite{RJ26} as an example for topologically mixing and Example \ref{ex3.1} showed that the operator is not only topologically mixing but also chaotic. 
Now we are ready to provide disjoint chaotic operators.

\begin{example}\textnormal{
Let $\mathcal{H}$ be a separable Hilbert space and $(e_{j})_{j\in\mathbb{Z}}\subset\mathcal{H}$ be an orthonormal basis for $\mathcal{H}$. From Example 1 in \cite{IT21}, we can take unitary operators $U^{(1)},U^{(2)}$ on $\mathcal{H}$ satisfying condition $(\ast)$.
And let $W_{1}$ and $W_{2}$ be bounded linear operators on $\mathcal{H}$ defined by}
\begin{displaymath}
W_{1}(e_{n})=\left\{ \begin{array}{ll}
2e_{n+1}, & \textrm{for } n<0,\\
\frac{1}{2}e_{n+1}, & \textrm{for } n\geq 0,
\end{array} \right.
\end{displaymath}
\begin{displaymath}
W_{2}(e_{n})=\left\{ \begin{array}{ll}
3e_{n+1}, & \textrm{for } n<0,\\
\frac{1}{3}e_{n+1}, & \textrm{for } n\geq 0.
\end{array} \right.
\end{displaymath}
\textnormal{
Now we define $W^{(1)}_{j}=W_{1}$, $W^{(2)}_{j}=W^{2}_{2}$ for every $j\in \mathbb{Z}$.
Then from Example 3.6 in \cite{RJ26}, the operators $T_{U^{(1)},W^{(1)}}$, $T_{U^{(2)},W^{(2)}}$ are disjoint topologically transitive. And from Example \ref{ex3.1} the operators $T_{U^{(1)},W^{(1)}}$, $T_{U^{(2)},W^{(2)}}$ are chaotic respectively.
This concludes $T_{U^{(1)},W^{(1)}}$, $T_{U^{(2)},W^{(2)}}$ are disjoint chaotic.
}
\end{example}

\subsection{Generalized translation operators on the $C^{\ast}$-algebra}
In this subsection we discuss the generalized translation operators on non-unitary $C^{\ast}$-algebras. For reader convenience, we first recall the definition of the generalized translation operator from \cite{I24}.
Let $\mathcal{A}$ be a non-unital $C^{\ast}$-algebra such that $\mathcal{A}$ is a closed two-sided ideal in a unital $C^{\ast}$-algebra $\mathcal{A}_{1}$. Let $\Phi$ be an isometric $\ast$-isomorphism of $\mathcal{A}_{1}$ such that $\Phi(\mathcal{A})=\mathcal{A}$. Assume that there exists a net $\{p_{\alpha}\}_{\alpha}\subset\mathcal{A}$ consisting of self-adjoint elements with $\|p_{\alpha}\|\leq 1$ for all $\alpha$ and such that $\{p^{2}_{\alpha}\}_{\alpha}$ is an approximate unit for $A$. Suppose in addition that for all $\alpha$ there exists some $N_{\alpha}\in\mathbb{N}$ such that $\Phi^{n}(p_{\alpha})p_{\alpha}=0$ for all $n\neq N_{\alpha}$. Let $b\in \mathcal{A}_{1}$ be an invertible element of $\mathcal{A}_{1}$ and $T_{\Phi, b}$ be the operator on $\mathcal{A}_{1}$ defined by 
\[T_{\Phi,b}(a)=b\Phi(a)\] 
for all $a\in \mathcal{A}_{1}$.

The inverse of $T_{\Phi,b}$, which we will denote by $S_{\Phi, b}$, is given by 
\[S_{\Phi,b}=\Phi^{-1}(b^{-1})\Phi^{-1}(a)\] 
for all $a\in\mathcal{A}_{1}$. 
Again, since $\Phi^{-1}(\mathcal{A})=\mathcal{A}$ and $\mathcal{A}$ is a two-sided
ideal in $\mathcal{A}_{1}$, we have that $S_{\Phi,b}(\mathcal{A})\subseteq\mathcal{A}$, hence $T_{\Phi,b}(\mathcal{A})=\mathcal{A}=S_{\Phi,b}(\mathcal{A})$.

In the following theorem we give a necessary and sufficient condition for generalized translation operators on the non-unitary $C^{\ast}$-algebra $\mathcal{A}$.

\begin{theorem}\label{the4.1}
Let $b_{1}$, $b_{2}$, . . . $b_{N}$ be invertible elements of $\mathcal{A}_{1}$ and $\Phi$ be an isometric $\ast$-isomorphism of $\mathcal{A}_{1}$. And let $\mathcal{F}$ be a finitely invariant Furstenberg family and $\{r_{k}\}^{N}_{k=1}\subset \mathbb{N}$ satisfy $0<r_{1}<r_{2}<...<r_{N}$. For $1\leq m\leq N$ we denote $T_{\Phi, b_{m}}:=T_{m}$. Then the following are equivalent.

$(1)$ $T^{r_{1}}_{1}$, $T^{r_{2}}_{2}$, ... ,$T^{r_{N}}_{N}$ are disjoint $\mathcal{F}$-transitive.

$(2)$ For every $p_{\alpha}$ there exist sequences $\{q_{n}\}_{n}$ and $\{d^{(1)}_{n}\}_{n}$, $\{d^{(2)}_{n}\}_{n}$, . . . , $\{d^{(N)}_{n}\}_{n}$ in $\mathcal{A}$ satisfying;
for every $\varepsilon>0$ there exists a set $F_{\varepsilon}\in\mathcal{F}$ such that  for any $n\in F_{\varepsilon}$ and

$(i)$ any $1\leq l\leq N$,
\[\|q_{n}-p^{2}_{\alpha}\|<\varepsilon, \|d^{(l)}_{n}-p^{2}_{\alpha}\|<\varepsilon,\]
\[\|\Phi^{-r_{l}n}(b_{l})\Phi^{-r_{l}n+1}(b_{l})\cdots \Phi^{-1}(b_{l})q_{n}\|<\varepsilon,\]
\[\|\Phi^{r_{l}n-1}(b^{-1}_{l})\Phi^{r_{l}n-2}(b^{-1}_{l})\cdots\Phi^{-1}(b^{-1}_{l})b^{-1}_{l}d^{(l)}_{n}\|<\varepsilon,\]
and

$(ii)$ any $1\leq l\neq s \leq N$,
\[\|\Phi^{r_{s}n-r_{l}n}(b_{l})\Phi^{r_{s}n-r_{l}n+1}(b_{l})\cdots\Phi^{r_{s}n-1}(b_{l})\Phi^{r_{s}n-1}(b^{-1}_{s})\Phi^{r_{s}n-2}(b^{-1}_{s})\cdots \Phi(b^{-1}_{s})b^{-1}_{s}d^{(s)}_{n}\|<\varepsilon.\]
\end{theorem}

\begin{proof}
First we show $(i)\Rightarrow(ii)$. Sinc $T^{r_{1}}_{1}$, $T^{r_{2}}_{2}$,..., $T^{r_{N}}_{N}$ are disjoint $\mathcal{F}$-transitive, there exists a sequence $\{a_{n}\}_{n}\subset \mathcal{A}$ such that 
\[(a_{n}, T^{r_{1}n}_{1}(a_{n}),..., T^{r_{N}n}_{N}(a))\stackrel{\mathcal F}{\longrightarrow} (p_{\alpha}, p_{\alpha},..., p_{\alpha}), \textrm{ as } n\rightarrow \infty.\]
In other words, for every $\varepsilon>0$ there exists a set $F_{\varepsilon}\in \mathcal{F}$ such that for any $n\in F_{\varepsilon}$ and any $1\leq l\leq N$, $\|a_{n}-p_{\alpha}\|<\varepsilon$, $\|T^{r_{l}n}_{l}a_{n}-p_{\alpha}\|<\varepsilon$.
Since $\mathcal{F}$ is finitely invariant, we may assume that for any $n\in F_{\varepsilon}$ $n>N_{\alpha}$. 

Now for $1\leq l\leq N$, we define 
\[q_{n}:=a_{n}p_{\alpha}, d^{(l)}_{n}:=b_{l}\Phi(b_{l})\cdots\Phi^{r_{l}n-1}(b_{l})\Phi^{r_{l}n}(a_{n})p_{\alpha}.\]
Then the sequences $\{q_{n}\}_{n}$, $\{d^{(l)}_{n}\}_{n}\subset \mathcal{A}$, since $\mathcal{A}$ is an ideal of $\mathcal{A}_{1}$. For any $n\in F_{\varepsilon}$ and any $1\leq l, s\leq N, l\neq s$,
\[\|q_{n}-p^{2}_{\alpha}\|=\|a_{n}p_{\alpha}-p^{2}_{\alpha}\|\leq \|a_{n}-p_{\alpha}\|<\varepsilon,\]
\begin{eqnarray}\begin{split}\nonumber
\|d^{(l)}_{n}-p^{2}_{\alpha}\|&=\|b_{l}\Phi(b_{l})\cdot\Phi^{r_{l}n-1}(b_{l})\Phi^{r_{l}n}(a_{n})p_{\alpha}-p^{2}_{\alpha}\|
\\
&=\|T^{r_{l}n}_{l}(a_{n})p_{\alpha}-p^{2}_{\alpha}\|\leq \|T^{r_{l}n}_{l}(a_{n})-p_{\alpha}\|
&<\varepsilon,
\end{split}
\end{eqnarray}

\begin{eqnarray}\begin{split}\nonumber
\|\Phi^{-r_{l}n}(b_{l})\Phi^{-r_{l}n+1}(b_{l})\cdots \Phi^{-1}(b_{l})q_{n}\|&=\|\Phi^{-r_{l}n}(b_{l})\Phi^{-r_{l}n+1}(b_{l})\cdots \Phi^{-1}(b_{l})a_{n}p_{\alpha}\|
\\
&=\|\Phi^{-r_{l}n}(b_{l}\Phi(b_{l})\cdots \Phi^{r_{l}n-1}(b_{l})\Phi^{r_{l}n}(a_{n})-p_{\alpha})p_{\alpha}\|
\\
&\leq \|b_{l}\Phi(b_{l})\cdots \Phi^{r_{l}n-1}(b_{l})\Phi^{r_{l}n}(a_{n})-p_{\alpha}\|=\|T^{r_{l}n}_{l}(a_{n})-p_{\alpha}\|
\\
&<\varepsilon,
\end{split}
\end{eqnarray}

\begin{eqnarray}\begin{split}\nonumber
\|\Phi^{r_{l}n-1}(b^{-1}_{l})\Phi^{r_{l}n-2}(b^{-1}_{l})\cdots\Phi(b^{-1}_{l})b^{-1}_{l}d^{(l)}_{n}\|&=\|\Phi^{r_{l}n}(a_{n})p_{\alpha}\|=\|\Phi^{r_{l}n}(a_{n}-p_{\alpha})p_{\alpha}\|
\\
&=\|a_{n}-p_{\alpha}\|<\varepsilon,
\end{split}
\end{eqnarray}
\begin{eqnarray}\begin{split}\nonumber
\|\Phi^{r_{s}n-r_{l}n}(b_{l})&\Phi^{r_{s}n-r_{l}n+1}(b_{l})\cdots\Phi^{r_{s}n-1}(b_{l})\Phi^{r_{s}n-1}(b^{-1}_{s})\Phi^{r_{s}n-2}(b^{-1}_{s})\cdots \Phi(b^{-1}_{s})b^{-1}_{s}d^{(s)}_{n}\|
\\
&=\|\Phi^{r_{s}n-r_{l}n}(b_{l})\Phi^{r_{s}n-r_{l}n+1}(b_{l})\cdots\Phi^{r_{s}n-1}(b_{l})\Phi^{r_{s}n}(a_{n})p_{\alpha}\|
\\
&=\|\Phi^{r_{s}n-r_{l}n}(b_{l}\Phi(b_{l})\cdots\Phi^{r_{l}n-1}(b_{l})\Phi^{r_{l}n}(a_{n})\Phi^{r_{l}n-r_{s}n}(p_{\alpha}))\|
\\
&=\|b_{l}\Phi(b_{l})\cdots \Phi^{r_{l}n-1}(b_{l}) \Phi^{r_{l}n}(a_{n})\Phi^{r_{l}n-r_{s}n}(p_{\alpha})\|
\\
&=\|T^{r_{l}n}_{l}(a_{n})\Phi^{r_{l}n-r_{s}n}(p_{\alpha})\|
\\
&=\|T^{r_{l}n}_{l}(a_{n})\Phi^{r_{l}n-r_{s}n}(p_{\alpha})-p_{\alpha}\Phi^{r_{l}n-r_{s}n}(p_{\alpha})\|
\\
&\leq\|T^{r_{l}n}_{l}(a_{n})-p_{\alpha}\|
\\
&<\varepsilon,
\end{split}
\end{eqnarray}
which concludes $(i)\Rightarrow(ii)$.

Now we show $(ii)\Rightarrow(i)$. Let $O, V_{1}, ..., V_{N}$ be non-empty subsets of $\mathcal{A}$. Since $\{p^{2}_{\alpha}\}_{\alpha}$ is an approximate unit for $\mathcal{A}$, there exist $p_{\alpha}$ and $x\in O$, $y_{1}\in V_{1}$, . . . , $y_{N}\in V_{N}$ such that $x=p^{2}_{\alpha}x$, $y_{1}=p^{2}_{\alpha}y_{1}$, . . . , $y_{N}=p^{2}_{\alpha}y_{N}$. 
Let us denote $\|x\|+\|y_{1}\|+\cdots+\|y_{N}\|=c$.

Define 
\[x_{n}:=q_{n}x+\sum^{N}_{l=1}S^{r_{l}n}_{l}(d^{(l)}_{n}y_{l}).\]
Now it is sufficient to show that 
\[(x_{n}, T^{r_{1}n}_{1}(x_{n}), ..., T^{r_{N}n}_{N}(x_{N}))\stackrel{\mathcal F}{\longrightarrow} (x, y_{1}, ..., y_{N}), \textrm{ as } n\rightarrow \infty.\] 
From condition (ii), there exist sequences of elements of $\mathcal{A}$ $\{q_{n}\}_{n}$ and $\{d^{(1)}_{n}\}_{n}$, $\{d^{(2)}_{n}\}_{n}$, . . . , $\{d^{(N)}_{n}\}_{n}$  satisfying;

for every $\varepsilon>0$ there exists a set $F_{\varepsilon}\in\mathcal{F}$ such that for any $n\in F_{\varepsilon}$ and any $1\leq l\neq s \leq N$,
\[\|q_{n}-p^{2}_{\alpha}\|<\varepsilon/c, \|d^{(l)}_{n}-p^{2}_{\alpha}\|<\varepsilon/c,\]
\[\|\Phi^{-r_{l}n}(b_{l})\Phi^{-r_{l}n+1}(b_{l})\cdots \Phi^{-1}(b_{l})q_{n}\|<\varepsilon/c,\]
\[\|\Phi^{r_{l}n-1}(b^{-1}_{l})\Phi^{r_{l}n-2}(b^{-1}_{l})\cdots\Phi^{-1}(b^{-1}_{l})b^{-1}_{l}d^{(l)}_{n}\|<\varepsilon/c,\]
\[\|\Phi^{r_{s}n-r_{l}n}(b_{l})\Phi^{r_{s}n-r_{l}n+1}(b_{l})\cdots\Phi^{r_{s}n-1}(b_{l})\Phi^{r_{s}n-1}(b^{-1}_{s})\Phi^{r_{s}n-2}(b^{-1}_{s})\cdots \Phi(b^{-1}_{s})b^{-1}_{s}d^{(s)}_{n}\|<\varepsilon/c.\]
Then 
\begin{eqnarray}\begin{split}\nonumber
\|x_{n}-x\|&=\|q_{n}x+\sum^{N}_{l=1}S^{r_{l}n}_{l}(d^{(l)}_{n}y_{l})-x\|
\\
&\leq \|q_{n}x-p^{2}_{\alpha}x\|+\sum^{N}_{l=1}\|S^{r_{l}n}_{l}(d^{(l)}_{n}y_{l})\|
\\
&\leq\|q_{n}-p^{2}_{\alpha}\|\|x\|+\sum^{N}_{l=1}\|S^{r_{l}n}_{l}(d^{(l)}_{n}y_{l})\|
\\
&=\|q_{n}-p^{2}_{\alpha}\|\|x\|+\sum^{N}_{l=1}\|\Phi^{-1} (b^{-1}_{l})\cdots \Phi^{-r_{l}n} (b^{-1}_{l})\Phi^{-r_{l}n} (d^{(l)}_{n}y_{l})\|
\end{split}
\end{eqnarray}

\begin{eqnarray}\begin{split}\nonumber
&=\|q_{n}-p^{2}_{\alpha}\|\|x\|+\sum^{N}_{l=1}\|\Phi^{r_{l}n}( \Phi^{r_{l}n-1}(b^{-1}_{l})\cdots \Phi(b^{-1}_{l})b^{-1}_{l}d^{(l)}_{n}y_{l})\|
\\
&\leq\|q_{n}-p^{2}_{\alpha}\|\|x\|+\sum^{N}_{l=1}\| \Phi^{r_{l}n-1}(b^{-1}_{l})\cdots \Phi(b^{-1}_{l})b^{-1}_{l}d^{(l)}_{n}\|\|y_{l}\|
\\
&<\frac{\varepsilon}{c}\|x\|+\sum^{N}_{l=1}\frac{\varepsilon}{c}\|y_{l}\|
\\
&=\varepsilon.
\end{split}
\end{eqnarray}
And
\begin{eqnarray}\begin{split}\nonumber
&\|T^{r_{l}n}_{l}(x_{n})-y_{l}\|
=\|T^{r_{l}n}_{l}(q_{n}x)+ T^{r_{l}n}_{l}(\sum^{N}_{s=1}S^{r_{s}N}_{s}(d^{(s)}_{n}y_{s}))-y_{l}\|
\\
&\leq \|T^{r_{l}n}_{l}(q_{n}x)\|+\sum^{N}_{s=1, s\neq l}T^{r_{l}n}_{l}(S^{r_{s}N}_{s}(d^{(s)}_{n}y_{s}))
\\
&=\|\Phi^{r_{l}n}(\Phi^{-r_{l}n}(b_{l})\Phi^{-r_{l}n+1}(b_{l})\cdots \Phi^{-1}(b_{l})q_{n}x)\
\\
&+\sum^{N}_{s=1, s\neq l}\|b_{l}\Phi(b_{l})\cdots\Phi^{r_{l}n-1}(b_{l}) \Phi^{r_{l}n}(\Phi^{-1}(b^{-1}_{s})\cdots\Phi^{-r_{s}n}(b^{-1}_{s})\Phi^{-r_{s}n}(d^{(s)}_{n}y_{s}))\|
\\
&=\|\Phi^{-r_{l}n}(b_{l})\Phi^{-r_{l}n+1}(b_{l})\cdots \Phi^{-1}(b_{l})q_{n}x \|
\\
&+\sum^{N}_{s=1, s\neq l}\|\Phi^{r_{l}n-r_{s}n}(\Phi^{r_{s}n-r_{l}n}(b_{l})\Phi^{r_{s}n-r_{l}n+1}(b_{l})\cdots\Phi^{r_{s}n-1}(b_{l})\Phi^{r_{s}n-1}(b^{-1}_{s})\Phi^{r_{s}n-2}(b^{-1}_{s})\cdots \Phi(b^{-1}_{s})b^{-1}_{s}d^{(s)}_{n}y_{s})\|
\\
&=\|\Phi^{-r_{l}n}(b_{l})\Phi^{-r_{l}n+1}(b_{l})\cdots \Phi^{-1}(b_{l})q_{n}x \|
\\
&+\sum^{N}_{s=1, s\neq l}\|\Phi^{r_{s}n-r_{l}n}(b_{l})\Phi^{r_{s}n-r_{l}n+1}(b_{l})\cdots\Phi^{r_{s}n-1}(b_{l})\Phi^{r_{s}n-1}(b^{-1}_{s})\Phi^{r_{s}n-2}(b^{-1}_{s})\cdots \Phi(b^{-1}_{s})b^{-1}_{s}d^{(s)}_{n}y_{s}\|
\\
&=\|\Phi^{-r_{l}n}(b_{l})\Phi^{-r_{l}n+1}(b_{l})\cdots \Phi^{-1}(b_{l})q_{n}x \|
\\
&+\sum^{N}_{s=1, s\neq l}\|\Phi^{r_{s}n-r_{l}n}(b_{l})\Phi^{r_{s}n-r_{l}n+1}(b_{l})\cdots\Phi^{r_{s}n-1}(b_{l})\Phi^{r_{s}n-1}(b^{-1}_{s})\Phi^{r_{s}n-2}(b^{-1}_{s})\cdots \Phi(b^{-1}_{s})b^{-1}_{s}d^{(s)}_{n}\|\|y_{s}\|
\\
&<\frac{\varepsilon}{c}\|x\|+\sum^{N}_{s=1, s\neq l}\frac{\varepsilon}{c}\|y_{s}\|
\\
&\leq \varepsilon.
\end{split}
\end{eqnarray}
This completes the proof.
\end{proof}

Using Theorem \ref{the4.1}, we can obtain an equivalent condition for the operators $T^{r_{1}}_{\Phi,b_{1}}$, $T^{r_{2}}_{\Phi,b_{2}}$, ... ,$T^{r_{N}}_{\Phi,b_{N}}$ to be disjoint topologically transitive, and
to be disjoint topologically mixing by taking the finitely invariant Furstenberg family $\mathcal{F}$, as $\mathcal{F}_{\textnormal{inf}}$ and $\mathcal{F}_{\textnormal{cof}}$ respectively.

And the condition (i) of (2) in Theorem \ref{the4.1} gives us an equivalet condition for the operator $T_{\Phi,b}$ to be $\mathcal{F}$-transitive as following; 

\begin{proposition}
Let $b$ be an invertible element of $\mathcal{A}_{1}$ and $\Phi$ be an isometric $\ast$-isomorphism of $\mathcal{A}_{1}$. And let $\mathcal{F}$ be a finitely invariant Furstenberg family. Then the following are equivalent.

$(1)$ $T_{\Phi, b}$ is $\mathcal{F}$-transitive.

$(2)$ For every $p_{\alpha}$ there exist sequences $\{q_{n}\}_{n}$ and $\{d_{n}\}_{n}$ in $\mathcal{A}$ such that for every $\varepsilon>0$ there exists a set $F_{\varepsilon}\in\mathcal{F}$ such that  for any $n\in F_{\varepsilon}$,
\[\|q_{n}-p^{2}_{\alpha}\|<\varepsilon, \|d_{n}-p^{2}_{\alpha}\|<\varepsilon,\]
\[\|\Phi^{-n}(b)\Phi^{-n+1}(b)\cdots \Phi^{-1}(b)q_{n}\|<\varepsilon,\]
\[\|\Phi^{n-1}(b^{-1})\Phi^{n-2}(b^{-1})\cdots\Phi(b^{-1})b^{-1}d_{n}\|<\varepsilon.\]
\end{proposition}

In \cite{IT21}, the authors characterized topological transitivity for translation operators $T_{U, W}$ (which is different from ours discussed in Section 3) on the $C^{\ast}$-algebra $B_{0}(\mathcal{H})$ of compact operators on the Hilbert space $\mathcal{H}$ and in \cite{IT23dis}, they characterized disjoint topological transitivity for the same operators. As mentioned in Example 4.3 of \cite{I24}, this operator can be represented as a generalized translation operator $T_{\Phi,b}$ on $C^{\ast}$-algebra $B_{0}(\mathcal{H})$ and Theorem \ref{the4.1} gives an equivalent condition for disjoint $\mathcal{F}$-transitivity. The following proposition characterize disjoint $\mathcal{F}$-transitivity for translation operators on the $C^{\ast}$-algebra $B_{0}(\mathcal{H})$ of compact operators on the Hilbert space $\mathcal{H}$.

\begin{proposition}
Let $\mathcal{H}$ be an Hilbert space and $W_{1}$, $W_{2}$, ..., $W_{N}$ be invertible operators on $\mathcal{H}$. And let $U$ be a unitary operator on $\mathcal{H}$ such that for every $k\in\mathbb{N}$ there exists $N_{k}$ with 
\[U^{n}(L_{m})\bot L_{m}, \textrm{ } n\geq N_{k}.\]
For $1\leq l\leq N$, let us define $T_{U, W_{l}}(F):=W_{l}FU$, for $F\in B_{0}(\mathcal{H})$. 
Assume that $\mathcal{F}$ is a finitely invariant Furstenberg family and $\{r_{k}\}^{N}_{k=1}\subset \mathbb{N}$ satisfies $0<r_{1}<r_{2}<...<r_{N}$.  Then the following are equivalent.

$(1)$ $T_{U, W_{1}}$, $T_{U, W_{2}}$, ..., $T_{U, W_{N}}$ are disjoint $\mathcal{F}$-transitive on $B_{0}(\mathcal{H})$.

$(2)$ For every $m \in\mathbb{N}$, there exist sequences
$\{D_{k}\}^{\infty}_{k=1}$ and $\{G^{(1)}_{k}\}^{\infty}_{k=1}$,...,$\{G^{(N)}_{k}\}^{\infty}_{k=1}\subset B_{0}(\mathcal{H})$ satisfying;

for every $\varepsilon>0$ there exists a set $F_{\varepsilon}\in\mathcal{F}$ such that  for any $n\in F_{\varepsilon}$ and any $1\leq l\neq s \leq N$,
\[\|D_{n}-P_{m}\|<\varepsilon, \|G^{(l)}_{n}-P_{m}\|<\varepsilon,\]
\[\|W^{r_{l}n}_{l}D_{n}\|<\varepsilon, \|W^{-r_{l}n}_{l}G^{(l)}_{n}\|<\varepsilon\]
and
\[\|W^{r_{l}n}_{l}W^{-r_{s}n}_{s}G^{(s)}_{n}\|<\varepsilon.\]
\end{proposition}

\end{document}